\documentclass[reqno,oneside]{amsart}
\usepackage{hyperref}
\usepackage{geometry}
\usepackage[ansinew]{inputenc}
\usepackage{graphicx}
\usepackage{amsmath}
\usepackage{amsthm}
\usepackage{amssymb, color}%
\usepackage[numbers, square]{natbib}
\usepackage{mathrsfs}
\usepackage{bbm}
\usepackage{tikz}
\usepackage{mathptmx}

\newcommand{\mmp}{\mathbb{P}}

\newcommand{\me}{\mathbb{E}}
\newcommand{\E}{\mathbb{E}}

\newcommand{\mr}{\mathbb{R}}
\newcommand{\mn}{\mathbb{N}}

\newcommand{\lit}{\underset{t\to\infty}{\lim}}

\DeclareMathOperator{\1}{\mathbbm{1}}

\newtheorem{thm}{Theorem}[section]
\newtheorem{lemma}[thm]{Lemma}

\newtheorem{cor}[thm]{Corollary}

\newtheorem{assertion}[thm]{Proposition}
\theoremstyle{definition}

\theoremstyle{remark}
\newtheorem{rem}[thm]{Remark}

\begin{document}

\title[Stable fluctuations of iterated perturbed random walks]{Stable fluctuations of iterated perturbed random walks in intermediate generations of a general branching process tree}

\author{Alexander Iksanov}
\address{Faculty of Computer Science and Cybernetics, Taras Shevchenko National University of Kyiv, Kyiv, Ukraine}
\email{iksan@univ.kiev.ua}

\author{Alexander Marynych}
\address{Faculty of Computer Science and Cybernetics, Taras Shevchenko National University of Kyiv, Kyiv, Ukraine}
\email{marynych@unicyb.kiev.ua}

\author{Bohdan Rashytov}
\address{Faculty of Computer Science and Cybernetics, Taras Shevchenko National University of Kyiv, Kyiv, Ukraine}
\email{mr.rashytov@gmail.com}

\begin{abstract}
Consider a general branching process, a.k.a.\ Crump-Mode-Jagers process, generated by a perturbed random walk $\eta_1$, $\xi_1+\eta_2$, $\xi_1+\xi_2+\eta_3,\ldots$. Here, $(\xi_1,\eta_1)$, $(\xi_2, \eta_2),\ldots$ are independent identically distributed random vectors with arbitrarily dependent positive components. Denote by $N_j(t)$ the number of the $j$th generation individuals with birth times $\leq t$. Assume that $j=j(t)\to\infty$ and $j(t)=o(t^a)$ as $t\to\infty$ for some explicitly given $a>0$ (to be specified in the paper). The corresponding $j$th generation belongs to the set of intermediate generations. We provide sufficient conditions under which finite-dimensional distributions of the process $(N_{\lfloor j(t)u\rfloor}(t))_{u>0}$, properly normalized and centered, converge weakly to those of an integral functional of a stable L\'{e}vy process with finite mean.
\end{abstract}



\maketitle

\section{Introduction and main result}\label{Sect1}

\subsection{Definition and motivation}

Let $(\xi_1, \eta_1)$, $(\xi_2, \eta_2),\ldots$ be independent copies of an $\mr^2$-valued random vector $(\xi, \eta)$ with arbitrarily dependent components. Denote by $(S_i)_{i\geq 0}$ the zero-delayed
standard random walk with increments $\xi_i$ for $i\in\mn$, that is, $S_0:=0$ and $S_i:=\xi_1+\cdots+\xi_i$ for $i\in\mn$. Define
\begin{equation*}
T_i:=S_{i-1}+\eta_i,\quad i\in \mn.
\end{equation*}
The sequence $T:=(T_i)_{i\in\mn}$ is called {\it perturbed random walk} (PRW). The so defined PRW is a non-trivial generalization of the standard random walk. Apart from being an interesting object of investigation, the PRW is known to be an important ingredient of perpetuities \cite{Goldie+Maller:2000}, the Bernoulli sieve \cite{Alsmeyer+Iksanov+Marynych:2017,Gnedin+Iksanov+Marynych:2010}, $G/G/\infty$-queues \cite{Iksanov+Jedidi+Bouzeffour:2018}, a perturbed branching random walk \cite{Basrak+Conroy+Cravioto+Palmowski:2022}, to name but a few. A detailed exposition of various PRW's properties and its applications can be found in \cite{Iksanov:2016}.

In what follows we assume that $\xi$ and $\eta$ are almost surely (a.s.) positive. Now we recall the construction of a general branching process generated by $T$. Imagine a population of individuals initiated at time $0$ by one individual, the ancestor. An individual born at time $s\geq 0$ produces offspring whose birth times have the same distribution as $(s+T_i)_{i\in\mn}$. All individuals act independently of each other. An individual resides in the $j$th generation if it has exactly $j$ ancestors. For $j\in\mn$ and $t\geq 0$, denote by $T^{(j)}$ the collection of the birth times in the $j$th generation and by $N_j(t)$ the number of the $j$th generation individuals with birth times $\leq t$. We call the sequence $\mathcal{T}:=(T^{(j)})_{j\in\mn}$ {\it iterated perturbed random walk on a general branching process tree}. Also, we call the $j$th generation {\it early}, {\it intermediate} or {\it late} depending on whether $j$ is fixed, $j=j(t)\to\infty$ and $j(t)=o(t)$ as $t\to\infty$, or $j=j(t)$ is of order $t$. According to Proposition 2.1 in \cite{Bohun etal:2022}, there exists a constant $a_0>0$ such that, for $j\geq at $, $a>a_0$ and large $t$, $N_j(t)=0$ a.s.

The sequence $\mathcal{T}$ has been introduced and used in \cite{Buraczewski+Dovgay+Iksanov:2020} (see also \cite{Iksanov+Marynych+Samoilenko:2020}) as an auxiliary tool in the analysis of the nested occupancy scheme in random environment generated by stick-breaking. Later on, it was realised that $\mathcal{T}$ was an interesting mathematical object on its own. Of particular interest is the question on how the properties of $T$ transform when passing to the early, the intermediate and then the late generations. Answering this question leads to a new generalization of renewal theory for the perturbed random walks, see \cite{Bohun etal:2022} and \cite{Iksanov+Rashytov+Samoilenko:2021+} for the first results in this direction. Although the iterated perturbed random walk on a general branching process tree is a particular instance of a branching random walk, we believe that understanding its properties provides some insight into the behavior of branching random walks and general branching processes with arbitrary (but admissible) inputs.

\subsection{Main result}

Throughout the paper we write 
$\Longrightarrow$, ${\overset{{\rm d}}\longrightarrow}$ and ${\overset{{\rm f.d.d.}}\longrightarrow}$ to denote weak convergence in a function space, weak convergence of one-dimensional and finite-dimensional distributions, respectively. The following result is a combination of Theorems 3.1 and 3.2 in \cite{Iksanov+Marynych+Samoilenko:2020}, see also Section 3 in \cite{Buraczewski+Dovgay+Iksanov:2020} for an earlier weaker version. Put
$$
V_j(t):=\E N_j(t),\quad  j\in\mn,\quad t\geq 0.
$$

\begin{assertion}\label{thm old}
Assume that ${\tt s}^2={\rm Var}\,\xi\in (0,\infty)$ and $\me \eta<\infty$. Let $j=j(t)$ be any positive integer-valued function satisfying $j(t)\to \infty$ and $j(t)=o(t^{1/2})$ as $t\to\infty$. Then, as $t\to\infty$,
\begin{equation}\label{clt22_old}
\left(\frac{\lfloor j(t)\rfloor^{1/2}(\lfloor j(t)u\rfloor-1)!}{({\tt s}^2{\tt m}^{-2\lfloor j(t)u\rfloor-1}t^{2\lfloor j(t)u\rfloor-1})^{1/2}}\bigg(N_{\lfloor j(t)u\rfloor}(t)-V_{\lfloor j(t)u \rfloor}(t)\bigg)\right)_{u>0}\\
{\overset{{\rm f.d.d.}}\longrightarrow}~ \Bigg(\int_{[0,\,\infty)}e^{-uy}{\rm d}\mathcal{S}_2(y)\Bigg)_{u>0},
\end{equation}
where ${\tt m}:=\me\xi<\infty$ and $\mathcal{S}_2:=(\mathcal{S}_2(v))_{v\geq 0}$ is a standard Brownian motion.
\end{assertion}

In this article we intend to prove a counterpart of Proposition \ref{thm old} under the assumptions that ${\tt m}<\infty$, ${\tt s}^2=\infty$ and that the distribution of $\xi$ belongs to the domain of attraction of a stable distribution. More precisely, we assume that one of the following conditions holds:

\noindent {\sc Condition RW I}: ${\tt s}^2=\infty$ and, for some $\ell$ slowly varying at infinity,
\begin{equation}    \label{eq:slowly varying 2nd moment}
\me (\xi^2 \1_{\{\xi \leq t\}})~\sim~\ell(t),\quad t\to\infty,
\end{equation}
in which case the distribution of $\xi$ belongs to the (non-normal) domain of attraction of a normal distribution, or

\noindent {\sc Condition RW II}: for some $\alpha\in (1,2)$ and some $\ell$ slowly varying at infinity,
\begin{equation}\label{eq:domain alpha}
\mmp\{\xi>t\}~\sim~t^{-\alpha}\ell(t),\quad t\to\infty,
\end{equation}
in which case the distribution of $\xi$ belongs to the domain of attraction
of an $\alpha$-stable distribution.

Assume that \eqref{eq:domain alpha} holds with $\alpha=1$. There exist slowly varying $\ell$ for which ${\tt m}<\infty$. Thus, in principle, this situation could have also been considered. However, we do not treat the case $\alpha=1$, for it is technically more complicated than the
others and does not shed any new light on weak convergence that we are interested in.

We shall write $N$ for $N_1$, that is, $N(t):=\sum_{i\geq 1}\1_{\{T_i\leq t\}}$ for $t\geq 0$. Denote by $D$ the Skorokhod space of right-continuous functions defined on $[0,\infty)$ with finite limits from the left at positive points. For later needs, we recall the following functional limit theorems, obtained in Theorem 3.2 of \cite{Alsmeyer+Iksanov+Marynych:2017}, for the process $(N(ty))_{y\geq 0}$ as $t\to\infty$: under the additional assumption $\me\eta^a<\infty$ for some $a>0$,
\begin{equation}    \label{eq:FLT for N(t)}
\Big(\frac{N(ty)-{\tt m}^{-1}\int_0^{ty}\mmp\{\eta\leq x\}{\rm d}x}{{\tt m}^{-1-1/\alpha}c_\alpha(t)}\Big)_{y\geq 0} ~\Longrightarrow~ (\mathcal{S}_\alpha(y))_{y\geq 0},\quad t\to\infty.
\end{equation}
Here,
\begin{itemize}
\item
under Condition RW I $\alpha=2$, $\mathcal{S}_2$ is a standard Brownian motion, $c_2(t)$ is a positive function satisfying
\begin{equation*}
\lim_{t\to\infty} t \ell(c_2(t)) / c^2_2(t)= 1,
\end{equation*}
and the convergence takes place in the $J_1$-topology on $D$;
\item under Condition RW II $\mathcal{S}_\alpha:=(\mathcal{S}_\alpha(u))_{u \geq 0}$ is a
spectrally negative $\alpha$-stable L\'{e}vy process such that
$\mathcal{S}_\alpha(1)$ has the characteristic function
\begin{equation}\label{char}
\mathbb{E} \exp({\rm i}z\mathcal{S}_\alpha(1))= \exp\{-|z|^\alpha
\Gamma(1-\alpha)(\cos(\pi\alpha/2)+{\rm i}\,{\rm sgn}(z)\sin(\pi\alpha/2))\},\quad z\in\mr,
\end{equation}
where $\Gamma(\cdot)$ denotes Euler's gamma function, $c_{\alpha}(t)$ is a positive function satisfying
\begin{equation*}
\lit t\ell(c_\alpha(t))/c^\alpha_\alpha(t)=1,
\end{equation*}
the convergence takes place in the $M_1$-topology on $D$.
\end{itemize}

Comprehensive information concerning the $J_1$- and $M_1$-convergence on $D$ can be found in the monographs \cite{Billingsley:1999, Jacod+Shiryaev:2003} and \cite{Whitt:2002}, respectively.

We recall that $\me \xi^\gamma<\infty$ for all $\gamma\in (0,\alpha)$ whenever either Condition RW I or RW II holds. Further, $\me \xi^2=\infty$ under Condition RW I. In contrast, $\me \xi^\alpha$ may be finite or infinite under Condition RW II. 
After this discussion we are ready to state assumptions on the distribution of $\eta$.

\noindent {\sc Condition Pert($\gamma$)}. If $\alpha\in(1,2)$ and $\me\xi^\alpha<\infty$, we set $\gamma:=\alpha$ and assume that
\begin{equation}\label{eq:gamma_def}
\me (\eta\wedge t)=O(t^{2-\gamma}),\quad t\to\infty.
\end{equation}
If $\alpha\in(1,2]$ and $\me\xi^\alpha=\infty$ we assume that \eqref{eq:gamma_def} holds for some $\gamma\in (2-1/\alpha,\alpha)$.

Here is our main result.
\begin{thm}\label{main100}
Assume that $\me\xi^2=\infty$, that the distribution of $\xi$ belongs to the domain of attraction of an $\alpha$-stable distribution, $\alpha\in (1,2]$ and that Condition {\sc Pert($\gamma$)} holds. Let $j=j(t)$ be any positive integer-valued function satisfying $j(t)\to \infty$ and $j(t)=o(t^{(\gamma-1)/2})$ as $t\to\infty$. Then, as $t\to\infty$,
\begin{equation}\label{clt22}
\left(\frac{(\lfloor j(t)u\rfloor-1)!{\tt m}^{\lfloor j(t)u\rfloor+1/\alpha}}{t^{\lfloor j(t)u\rfloor-1}c_\alpha(t/j(t))}\bigg(N_{\lfloor j(t)u\rfloor}(t)-V_{\lfloor j(t)u \rfloor}(t)\bigg)\right)_{u>0}\\
{\overset{{\rm f.d.d.}}\longrightarrow}~ \Bigg(\int_{[0,\,\infty)}e^{-uy}{\rm d}\mathcal{S}_\alpha(y)\Bigg)_{u>0},
\end{equation}
where ${\tt m}=\me\xi<\infty$, $\mathcal{S}_2$ is a standard Brownian motion and $\mathcal{S}_\alpha$ is a
spectrally negative $\alpha$-stable L\'{e}vy process with characteristic function \eqref{char}.
\end{thm}
\begin{rem}
One can check that the inequality $\me \eta^{\gamma-1}<\infty$ ensures Condition {\sc Pert($\gamma$)}, and that Condition {\sc Pert($\gamma$)} guarantees that $\me \eta^{\gamma-1-\delta}<\infty$ for any $\delta\in (0,\gamma-1)$. The latter means that under the assumptions of Theorem \ref{main100} relation \eqref{eq:FLT for N(t)} holds.
\end{rem}
\begin{rem}
The limit process in Theorem \ref{main100}, that we denote by $L_\alpha$, is actually defined as the result of integration by parts:
$$
L_\alpha(u)=u\int_0^{\infty} e^{-uy}\mathcal{S}_\alpha(y){\rm d}y,\quad u>0.
$$ One can check that this definition produces the same process as an alternative definition appearing in Theorem \ref{main100} in which $L_\alpha$ is understood as the stochastic integral with the integrator being a semimartingale. Note that the process $L_\alpha$ is a.s.\ continuous and self-similar with {\it negative} index $-1/\alpha$, that is, for any $a>0$, any $r\in\mn$ and any $0<u_1<\ldots<u_r<\infty$, the vector $(L_\alpha(au_1),\ldots, L_\alpha(au_r))$ has the same distribution as $a^{-1/\alpha}(L_\alpha(u_1),\ldots, L_\alpha(u_r))$.

Assume that $\alpha\in (1,2)$. The process $\mathcal{S}_\alpha$ which describes the limit fluctuations of $N_1=N$ (in the first generation) is a.s.\ discontinuous. The structure of the process $L_\alpha$ indicates that the limit fluctuations of $N_j$ (in the intermediate generations $j$) are driven by two factors: (i) the fluctuations of the input process $N_1$ which are governed by $\mathcal{S}_\alpha$; (ii) the renewal structure of the tree which is reflected in the function $u\mapsto e^{-uy}$. Furthermore, we see that the renewal structure of the tree makes the limit $L_\alpha$ continuous, thereby smoothing out the fluctuations of the input process.
\end{rem}

The remainder of the paper is structured as follows. Some auxiliary results are stated and proved in Section \ref{auxstat}. The proof of Theorem \ref{main100} is given in Section \ref{sect:flt}.

\section{Auxiliary results}\label{auxstat}

The Lebesgue--Stieltjes convolution of functions $r,s:[0,\infty)\to [0,\infty)$ of locally bounded variation is given by
$$
(r\ast s)(t) = \int_{[0,\,t]} r(t-y){\rm d}s(y)=\int_{[0,\,t]} s(t-y){\rm d}r(y),\quad t\geq 0.
$$
We write $r^{\ast (j)}$ for the $j$-fold Lebesgue--Stieltjes convolution of $r$ with itself.

We proceed by recalling an extended version of Proposition 3.1 in \cite{Bohun etal:2022}. Inequality \eqref{ineq} is not a part of the cited result, it is contained in its proof.
\begin{lemma}\label{prop:convolutions1}
Let $f:\mr\to [0,\,\infty)$ be a nondecreasing right-continuous function vanishing on the negative half-line and satisfying
\begin{equation}\label{eq:asymp_exp_assump}
f(t)=at+O(t^{\beta}),\quad t\to\infty
\end{equation}
for some $a>0$ and $\beta\in [0,1)$. Then, for some constant $C\geq 1$,
\begin{equation}\label{ineq}
\Big|f^{\ast(j)}(t)-\frac{a^j t^j}{j!}\Big|\leq \sum_{i=0}^{j-1}\binom{j}{i}\frac{a^i C^{j-i}(t+1)^{\beta(j-i)+i}}{i!},\quad j\in\mn,~~ t\geq 0.
\end{equation}
In particular, for any integer-valued function $j=j(t)$ satisfying $j(t)=o(t^{(1-\beta)/2})$ as $t\to\infty$,
$$
f^{\ast(j)}(t)~\sim~\frac{a^j t^j}{j!},\quad t\to\infty.
$$
\end{lemma}

Of principal importance for what follows is the decomposition:
\begin{equation}\label{basic1232}
N_j(t)=\sum_{k\geq 1}N^{(k)}_{j-1}(t-T_k)\1_{\{T_k\leq t\}},
\quad j\geq 2,\quad t\geq 0,
\end{equation}
where $N_{j-1}^{(r)}(t)$ is the number of successors in the $j$th generation with birth times within $[T_r,t+T_r]$ of the first generation individual with birth time $T_r$.
In what follows, we write $V$ for $V_1$. Note that~\eqref{basic1232} entails $\me N_j(t)=V_j(t)=V^{\ast(j)}(t)$ for $j\in\mn$ and $t\geq 0$.

Corollary \ref{mean1} is our important technical tool to be used in all subsequent proofs.
\begin{cor}\label{mean1}
Assume that the assumptions of Theorem \ref{main100} hold. Then, for some constant $C\geq 1$,
\begin{equation}\label{ineq56789}
\Big|V_j(t)-\frac{t^j}{j!{\tt m}^j}\Big|\leq \sum_{i=0}^{j-1}\binom{j}{i}\frac{C^{j-i}(t+1)^{(2-\gamma)(j-i)+i}}{i!{\tt m}^i},\quad j\in\mn,\quad t\geq 0
\end{equation}
with the same $\gamma$ as in Condition {\sc Pert($\gamma$)}. In particular, for any integer-valued function $j=j(t)$ satisfying $j(t)=o(t^{(\gamma-1)/2})$ as $t\to\infty$,
\begin{equation}\label{asymptot}
V_j(t)~\sim~\frac{t^j}{j!{\tt m}^j},\quad t\to\infty.
\end{equation}

Let $j\in\mn$ and $s\geq 0$ satisfy $(s+1)^{\gamma-1}\geq 2C{\tt m}j^2$. Then, for $1\leq k\leq j$,
\begin{equation}\label{vk}
V_k(s)\leq \frac{2(s+1)^k}{k!{\tt m}^k},
\end{equation}
\begin{equation}\label{basic123}
\sum_{i=0}^{k-1}\binom{k}{i}\frac{C^{k-i}(s+1)^{(2-\gamma)(k-i)+i}}{i!{\tt m}^i}\leq \frac{2Ck(s+1)^{k+1-\gamma}}{(k-1)!{\tt m}^{k-1}}.
\end{equation}
and
\begin{equation}\label{basic12312}
\sum_{i=0}^{k-1}\binom{k}{i}\frac{C^{k-i}(s+1)^{(2-\gamma)(k-i)+i+1}}{(i+1)!{\tt m}^{i+1}}\leq \frac{2C(s+1)^{k+2-\gamma}}{(k-1)!{\tt m}^k}.
\end{equation}
\end{cor}
\begin{proof}
We shall show that the function $V$ satisfies the assumptions of Lemma \ref{prop:convolutions1} with $a={\tt m}^{-1}$ and $\beta=2-\gamma$. Then \eqref{ineq56789} and \eqref{asymptot} are
an immediate consequence of Lemma \ref{prop:convolutions1}.

Let $S_0^\ast$ be a random variable with distribution
$$
\mmp\{S_0^\ast\in{\rm d}x\}={\tt m}^{-1}\mmp\{\xi>x\}\1_{(0,\infty)}(x){\rm d}x.
$$
Then, according to formula (2) in \cite{Carlsson+Nerman:1986},
$$U(t)-{\tt m}^{-1}t=\int_{[0,\,t]}\mmp\{S_0^\ast>t-y\}{\rm d}U(y),\quad t\geq 0,$$ where $U(t):=\sum_{i\geq 0}\mmp\{S_i\leq t\}$ for $t\geq 0$, that is, $U$ is the renewal function of $(S_i)_{i\in\mn_0}$. Since the assumption $\me\xi^2=\infty$ is equivalent to $\me S^\ast=\infty$, we conclude that
$$U(t)-{\tt m}^{-1}t~\sim~ {\tt m}^{-1} \int_0^t \mmp\{S_0^\ast>y\}{\rm d}y,\quad t\to\infty$$ by Theorem 4 in \cite{Sgibnev:1981}.

Assume that $\me\xi^\alpha<\infty$. Then $\gamma=\alpha\in (1,2)$, $\me (S_0^\ast)^{\alpha-1}<\infty$ and, by Markov's inequality, $$\int_0^t \mmp\{S_0^\ast>y\}{\rm d}y\leq (2-\alpha)^{-1}\me (S_0^\ast)^{\alpha-1}t^{2-\alpha}.$$ This together with Condition {\sc Pert$(\gamma)$} which reads $\me (\eta\wedge t)=O(t^{2-\alpha})$ entails $$V(t)-{\tt m}^{-1}t=\int_{[0,\,t]}(U(t-y)-{\tt m}^{-1}(t-y)){\rm d}\mmp\{\eta\leq y\}-{\tt m}^{-1}\me (\eta\wedge t)=O(t^{2-\alpha}),\quad t\to\infty.$$

Assume that $\me\xi^\alpha=\infty$. Since $\me\xi^\gamma_1<\infty$, hence $\me (S_0^\ast)^{\gamma_1-1}<\infty$ for all $\gamma_1\in (0,\alpha)$, the same reasoning as above leads to the conclusion
$$
\int_{[0,\,t]}(U(t-y)-{\tt m}^{-1}(t-y)){\rm d}\mmp\{\eta\leq y\}=O(t^{2-\gamma_1}),\quad t\to\infty.
$$
In conjunction with \eqref{eq:gamma_def} this yields $V(t)-{\tt m}^{-1}t=O(t^{2-\gamma})$. In particular, there exists a constant $c>0$ such that
\begin{equation}\label{lord2}
|V(t)-{\tt m}^{-1}t|\leq c (t+1)^{2-\gamma},\quad t\geq 0.
\end{equation}

Next, we prove \eqref{vk}. According to \eqref{ineq56789}, it is enough to check that
$$\sum_{i=0}^{k-1}\binom{k}{i}\frac{C^{k-i}(s+1)^{(2-\gamma)(k-i)+i}}{i!{\tt m}^i}\leq \frac{(s+1)^k}{k!{\tt m}^k},\quad 1\leq k\leq j,\quad (s+1)^{\gamma-1}\geq 2C{\tt m}j^2.$$
Using
\begin{equation}\label{eq:bin}
\binom{k}{i}\leq \frac{k!}{i!}\leq k^{k-i}
\end{equation}
and
\begin{equation}\label{eq:aux}
(s+1)^k=(s+1)^{(2-\gamma)k}(s+1)^{(\gamma-1)k},
\end{equation}
this follows from
\begin{eqnarray}\label{aux1}
\frac{k!{\tt m}^k}{(s+1)^k}\sum_{i=0}^{k-1}\binom{k}{i}\frac{C^{k-i}(s+1)^{(2-\gamma)(k-i)+i}}{i!{\tt m}^i}&=&\sum_{i=0}^{k-1}\binom{k}{i}\frac{k!}{i!}\left(\frac{C{\tt m}}{(s+1)^{\gamma-1}}\right)^{k-i} \leq \sum_{i=0}^{k-1}\left(\frac{C{\tt m}k^2}{(s+1)^{\gamma-1}}\right)^{k-i} \notag\\&\leq& \sum_{i=0}^{k-1}\left(\frac{C{\tt m}k^2}{2C{\tt m}j^2}\right)^{k-i}=\sum_{i=1}^{k}\left(\frac{k^2}{2j^2}\right)^{i}\leq \sum_{i=1}^{\infty}2^{-i}=1
\end{eqnarray}
because $k\leq j$.

Now we are passing to the proof of \eqref{basic123}. Invoking once again \eqref{eq:bin} and \eqref{eq:aux} we arrive at
\begin{eqnarray*}
\frac{(k-1)!{\tt m}^{k-1}}{k(s+1)^{k+1-\gamma}}\sum_{i=0}^{k-2}\binom{k}{i}\frac{C^{k-i}(s+1)^{(2-\gamma)(k-i)+i}}{i!{\tt m}^i}&=& \frac{(s+1)^{\gamma-1}}{{\tt m}k^2}\sum_{i=0}^{k-2}\binom{k}{i}\frac{k!}{i!}\left(\frac{C{\tt m}}{(s+1)^{\gamma-1}}\right)^{k-i}\\&\leq& \frac{(s+1)^{\gamma-1}}{{\tt m}k^2}\sum_{i=0}^{k-2}\left(\frac{C{\tt m}k^2}{(s+1)^{\gamma-1}}\right)^{k-i}\\&\leq& \frac{(s+1)^{\gamma-1}}{{\tt m}k^2} \sum_{i\geq 2} \left(\frac{C{\tt m}k^2}{(s+1)^{\gamma-1}}\right)^i \\&=& \frac{{\tt m}(Ck)^2}{(s+1)^{\gamma-1}}\Big(1-\frac{C{\tt m}k^2}{(s+1)^{\gamma-1}}\Big)^{-1}\leq C,
\end{eqnarray*}
and \eqref{basic123} follows. The proof of \eqref{basic12312} is analogous, hence omitted. The proof of the corollary is complete.

%
\end{proof}

Lemma \ref{import} will be used in the proof of relation \eqref{clt22_sn} below.
\begin{lemma}\label{import}
Let $u>0$ be fixed. Under the assumptions of Theorem \ref{main100},
\begin{equation}\label{impo121}
\lim_{t\to\infty} \frac{(\lfloor j(t)u\rfloor-1)!{\tt m}^{\lfloor j(t)u\rfloor-1}}{t^{\lfloor j(t)u\rfloor-1}}V_{\lfloor j(t)u\rfloor -1}(t(1-y/j))=~e^{-uy}
\end{equation}
for each fixed $y\geq 0$, and
\begin{equation}\label{aux0}
\lim_{T\to\infty}\limsup_{t\to\infty}\frac{(\lfloor j(t)u\rfloor-1)!{\tt m}^{\lfloor j(t)u\rfloor}}{t^{\lfloor j(t)u\rfloor-1}c_\alpha(t/j)}\int_{(Tt/j,\,t]}c_\alpha(y){\rm d}_y(-V_{\lfloor j(t)u\rfloor-1}(t-y))=0.
\end{equation}
\end{lemma}
\begin{proof}
For notational simplicity, we only treat the case $u=1$. We first prove \eqref{impo121}. According to \eqref{ineq56789},
$$\Big|V_j(t)-\frac{t^j}{{\tt m}^j j!}\Big|\leq g_j(t), \quad j\in\mn,~~ t\geq 0,$$
where $$g_j(t):=\sum_{i=0}^{j-1}\binom{j}{i}\frac{C^{j-i}(t+1)^{(2-\gamma)(j-i)+i}}{{\tt m}^i i!},\quad j\in\mn,~~ t\geq 0.$$ It suffices to prove that, for each fixed $y>0$,
$$\lim_{t\to\infty} \frac{(j-1)!{\tt m}^{j-1}}{t^{j-1}}\frac{(t(1-y/j))^{j-1}}{{\tt m}^{j-1}(j-1)!}=e^{-y}$$ and
\begin{equation}\label{ineq122}
\lim_{t\to\infty} \frac{(j-1)!{\tt m}^{j-1}}{t^{j-1}}g_{j-1}(t(1-y/j))=0.
\end{equation}
The first of these is immediate. To prove the second, we first recall that $j(t)=o(t^{(\gamma-1)/2})$ as $t\to\infty$. Hence, for $t$ large enough, $\frac{C{\tt m}(j-1)^2}{(t+1)^{\gamma-1}}\leq 1/2$, say. Write, for such $t$, with the help of \eqref{aux1}
\begin{equation*}
\frac{(j-1)!{\tt m}^{j-1}}{(t+1)^{j-1}} g_{j-1}(t)\leq \sum_{i=0}^{j-2}\left(\frac{C{\tt m}(j-1)^2}{(t+1)^{\gamma-1}}\right)^{j-1-i}\leq \frac{C{\tt m}(j-1)^2}{(t+1)^{\gamma-1}}\left(1-\frac{C{\tt m}(j-1)^2}{(t+1)^{\gamma-1}}\right)^{-1}.
\end{equation*}
Since $$\frac{(j-1)^2}{(t(1-y/j)+1)^{\gamma-1}}~\sim~ \frac{j^2}{t^{\gamma-1}}~\to~ 0,\quad t\to\infty,$$ the last inequality entails \eqref{ineq122}.

Next, we intend to prove \eqref{aux0}. The function $c_\alpha$ is regularly varying at infinity of index $1/\alpha$, see, for instance, Lemma 6.1.3 in \cite{Iksanov:2016}. By Theorem 1.8.3 in \cite{BGT} and its proof, there exists an infinitely differentiable function $g_\alpha$ with nonincreasing derivative $g_\alpha^\prime$ which varies regularly at infinity of index $1/\alpha-1$. Without loss of generality, we can and do assume that $c_\alpha$ itself enjoys all these properties. As a consequence,
\begin{equation}\label{monotone}
\lim_{t\to\infty}\frac{tc_\alpha^\prime(t)}{c_\alpha(t)}=\frac{1}{\alpha}.
\end{equation}

Integrating by parts we infer $$\int_{(Tt/j,\,t]}c_\alpha(y){\rm d}_y(-V_{j-1}(t-y))=V_{j-1}(t(1-T/j))c_\alpha(Tt/j)+\int_{Tt/j}^t V_{j-1}(t-y)c_\alpha^\prime(y){\rm d}y.$$ In view of \eqref{impo121}, $$\lim_{t\to\infty}\frac{(j-1)!{\tt m}^j}{t^{j-1}c_\alpha(t/j)}V_{j-1}(t(1-T/j))c_\alpha(Tt/j)={\tt m}T^{1/\alpha}e^{-T}.$$
The right-hand side converges to $0$ as $T\to\infty$. Using \eqref{ineq56789} we obtain
\begin{multline*}
\frac{(j-1)!{\tt m}^j}{t^{j-1}c_\alpha(t/j)} \int_{Tt/j}^t V_{j-1}(t-y)c_\alpha^\prime(y){\rm d}y\leq \frac{{\tt m}}{t^{j-1}c_\alpha(t/j)}\int_{Tt/j}^t (t-y)^{j-1}{\rm d}c_\alpha(y)\\+\frac{(j-1)!{\tt m}^j}{t^{j-1}c_\alpha(t/j)}\sum_{i=0}^{j-2}\binom{j-1}{i}\frac{C^{j-1-i}}{{\tt m}^i i!}
\int_{Tt/j}^t (t+1-y)^{(2-\gamma)(j-1-i)+i}c_\alpha^\prime(y){\rm d}y =:a_j(t)+b_j(t).
\end{multline*}
Since $\lim_{t\to\infty}(c_\alpha(ty/j)/c_\alpha(t/j))=y^{1/\alpha}$ for each $y>0$, we infer
\begin{multline*}
a_j(t)=\frac{{\tt m}}{c_\alpha(t/j)} \int_T^j (1-y/j)^{j-1}{\rm d}_y c_\alpha(ty/j)\leq \frac{{\tt m}}{c_\alpha(t/j)} \int_T^j e^{-(j-1)y/j}{\rm d}_y c_\alpha(ty/j)\\
\leq \frac{{\tt m}}{c_\alpha(t/j)} \int_T^j e^{-y/2}{\rm d}_y c_\alpha(ty/j)\to \frac{{\tt m}}{\alpha}\int_T^\infty e^{-y/2}y^{1/\alpha-1}{\rm d}y,\quad t\to\infty.
\end{multline*}
Here, the limit relation is justified by the continuity theorem for Laplace-Stieltjes trasnforms.
The the right-hand side of the last centered formula converges to $0$ as $T\to\infty$. We claim that $\lim_{t\to\infty}\,b_j(t)=0$. To prove this, we first observe that
\begin{multline*}
\frac{1}{c_\alpha^\prime(Tt/j)} \int_{Tt/j}^t (t+1-y)^{(2-\gamma)(j-1-i)+i}c_\alpha^\prime(y){\rm d}y\leq \int_{Tt/j}^t (t+1-y)^{(2-\gamma)(j-1-i)+i}{\rm d}y\\=\frac{(t(1-T/j)+1)^{(2-\gamma)(j-1-i)+i+1}-1}{(2-\gamma)(j-1-i)+i+1}\leq \frac{t^{(2-\gamma)(j-1-i)+i+1}}{i+1},
\end{multline*}
where the first inequality follows from the fact that $c_\alpha^\prime$ is nonincreasing, and the last inequality holds for $t$ so large that $Tt/j\geq 1$ and, as a consequence, $t(1-T/j)+1\leq t$. Further, in view of \eqref{monotone}, $$\lim_{t\to\infty}\frac{(t/j)c_\alpha^\prime(Tt/j)}{c_\alpha(t/j)}=\alpha^{-1}T^{1/\alpha-1}.$$ Hence, for large $t$ and some constant $A(T)>0$, $$\frac{(t/j)c_\alpha^\prime(Tt/j)}{c_\alpha(t/j)}\leq A(T).$$ With these at hand, we infer, for large $t$,
\begin{multline*}
b_j(t)\leq \frac{(t/j)c_\alpha^\prime(t/j)}{c_\alpha(t/j)} \frac{j!{\tt m}^j}{t^j}\sum_{i=0}^{j-2}\binom{j-1}{i}\frac{C^{j-1-i}t^{(2-\gamma)(j-1-i)+i+1}}{(i+1)!{\tt m}^i}\leq A(T){\tt m}\sum_{i=0}^{j-2}\Big(\frac{C{\tt m}j^2}{t^{\gamma-1}}\Big)^{j-1-i}\\ \leq A(T) \frac{C{\tt m}^2j^2}{t^{\gamma-1}}\Big(1-\frac{C{\tt m}j^2}{t^{\gamma-1}}\Big)^{-1}~\to~0,\quad t\to\infty.
\end{multline*}
We have used \eqref{eq:bin} for the second inequality.
\end{proof}

\section{Proof of Theorem \ref{main100}}\label{sect:flt}

\subsection{Preparation}

We shall use a decomposition of $N_j-V_j$ into a `martingale' part and a `shot-noise' part obtained with the help of \eqref{basic1232}:
\begin{multline*}
N_j(t)-V_j(t)=\Big(\sum_{k\geq 1}(N^{(k)}_{j-1}(t-T_k)-V_{j-1}(t-T_k))\1_{\{T_k\leq t\}}\Big)\\
+ \Big(\sum_{k\geq 1}V_{j-1}(t-T_k)\1_{\{T_k\leq t\}}-V_j(t)\Big),\quad j\geq 2,\quad t\geq 0.
\end{multline*} We shall prove that, as $t\to\infty$,
\begin{equation}\label{clt21}
\frac{(\lfloor j(t)u\rfloor-1)!{\tt m}^{\lfloor j(t)u\rfloor}}{t^{\lfloor j(t)u\rfloor-1}c_\alpha(t/j(t))}\sum_{k\geq 1}\big(N^{(k)}_{\lfloor j(t)u\rfloor-1}(t)-V_{\lfloor j(t)u\rfloor-1}(t-T_k)\big)\1_{\{T_k\leq t\}}~{\overset{{\rm f.d.d.}}\longrightarrow}~(\Theta(u))_{u>0},
\end{equation}
where $\Theta(u):=0$ for $u>0$,
and
\begin{multline}\label{clt22_sn}
\left(\frac{(\lfloor j(t)u\rfloor-1)!{\tt m}^{\lfloor j(t)u\rfloor+1/\alpha}}{t^{\lfloor j(t)u\rfloor-1}c_\alpha(t/j(t))}\bigg(\sum_{k\geq
1}V_{\lfloor j(t)u\rfloor-1}(t-T_k)\1_{\{T_k\leq t\}}-V_{\lfloor j(t)u \rfloor}(t)\bigg)\right)_{u>0}\\
{\overset{{\rm f.d.d.}}\longrightarrow}~ \Bigg(\int_{[0,\,\infty)}e^{-uy}{\rm d}\mathcal{S}_\alpha(y)\Bigg)_{u>0},
\end{multline}
thereby showing that the asymptotics in focus is driven by the `shot-noise' part, whereas the contribution of the `martingale' part is negligible.

We start with several preparatory results which are needed for the proof of \eqref{clt22_sn}. Lemma~\ref{lem:flt} is a version of limit relation~\eqref{eq:FLT for N(t)} with a different centering.
\begin{lemma}\label{lem:flt}
Under the assumptions and notation of Theorem \ref{main100}, as $t\to\infty$,
\begin{equation}\label{eq:flt}
\Big(\frac{N(ut)-V(ut)}{{\tt m}^{-(\alpha+1)/\alpha}c_\alpha(t)}\Big)_{u\geq 0}~\Longrightarrow~ (\mathcal{S}_\alpha(u))_{u\geq 0}
\end{equation}
in the $J_1$-topology on $D$ if $\alpha=2$ and in the $M_1$-topology on $D$ if $\alpha\in (1,2)$.
\end{lemma}
\begin{proof}
Put $\nu(t):=\#\{k\in\mn_0:S_k\leq t\}$ for $t\geq 0$, so that $U(t)=\me\nu(t)$. According to Wald's identity, $U(t)={\tt m}^{-1}\E S_{\nu(t)}\geq {\tt m}^{-1}t$ for $t\geq 0$. It is shown in the proof of Corollary \ref{mean1} (see a few lines preceding \eqref{lord2}) that
\begin{equation}\label{eq:star}
U(t)-{\tt m}^{-1}t=O(t^{2-\gamma}),\quad t\to\infty.
\end{equation}
As a consequence, $$0\leq V(t)-{\tt m}^{-1}\int_0^t \mmp\{\eta\leq y\}{\rm d}y=\int_{[0,\,t]}(U(t-y)-{\tt m}^{-1}(t-y)){\rm d}\mmp\{\eta\leq y\}=O(t^{2-\gamma}),\quad t\to\infty.$$ Hence, relation \eqref{eq:flt} follows from \eqref{eq:FLT for N(t)} if we can show that
\begin{equation}\label{eq:relat}
\lim_{t\to\infty}\frac{t^{2-\gamma}}{c_\alpha(t)}=0.
\end{equation}

To prove \eqref{eq:relat}, recall that the function $c_\alpha$ is regularly varying at infinity of index $1/\alpha$ and that the $\gamma$ appearing in Condition {\sc Pert($\gamma$)} satisfies $\gamma\in (2-1/\alpha,\alpha]$. Thus, $2-\gamma<1/\alpha$. This justifies \eqref{eq:relat} and completes the proof of Lemma \ref{lem:flt}.
\end{proof}
\begin{lemma}\label{aux123}
Under the assumptions and notation of Theorem \ref{main100},
\begin{equation*}\label{230}
\lim_{t\to\infty} \me|N(t)-V(t)|/c_\alpha(t)={\tt m}^{-(\alpha+1)/\alpha}\me |\mathcal{S}_\alpha(1)|.
\end{equation*}
\end{lemma}
\begin{proof}
Putting $u=1$ in \eqref{eq:flt} yields
\begin{equation}\label{aux100}
\frac{N(t)-V(t)}{{\tt m}^{-(\alpha+1)/\alpha}c_\alpha(t)}~{\overset{{\rm d}}\longrightarrow}~ \mathcal{S}_\alpha(1),\quad t\to\infty.
\end{equation}
Fix any $r\in (1,\alpha)$. Assume that we can show that
\begin{equation}\label{inter}
\me |N(t)-V(t)|^r=O((c_\alpha(t))^r),\quad t\to\infty.
\end{equation}
Then the family $((N(t)-V(t))/c_\alpha(t))_{t\geq 1}$ is uniformly integrable. This together with \eqref{aux100} is sufficient for completing the proof.

\noindent {\sc Proof of \eqref{inter}}. We shall use a decomposition $$N(t)-V(t)=\sum_{k\geq 0}(\1_{\{S_k+\eta_{k+1}\leq t\}}-G(t-S_k))+\int_{[0,\,t]}G(t-x){\rm d}(\nu(x)-U(x)),$$ where $G(x):=\mmp\{\eta\leq x\}$ for $x\geq 0$. In view of $$|x+y|^r\leq 2^{r-1}(|x|^r+|y|^r),\quad x,y\in\mr,$$ it suffices to check that
\begin{equation}\label{11}
\me \Big|\sum_{k\geq 0}(\1_{\{S_k+\eta_{k+1}\leq t\}}-G(t-S_k))\Big|^r=O((c_\alpha(t))^r),\quad t\to\infty
\end{equation}
and
\begin{equation}\label{12}
D(t):=\me \Big|\int_{[0,\,t]}G(t-x){\rm d}(\nu(x)-U(x))\Big|^r=O((c_\alpha(t))^r),\quad t\to\infty.
\end{equation}

We first prove \eqref{11}. By Jensen's inequality, $(\me |X|^r)^{1/r}\leq (\me X^2)^{1/2}$ for any real-valued random variable $X$. Thus, \eqref{11} follows if we can check that
$$\me \Big(\sum_{k\geq 0}(\1_{\{S_k+\eta_{k+1}\leq t\}}-G(t-S_k))\Big)^2=O((c_\alpha(t))^2),\quad t\to\infty.$$ Actually, we shall prove even more, namely, that the right-hand side is $O(c_\alpha(t))$. The last expectation is equal to $$\int_{[0,\,t]}G(t-x)(1-G(t-x)){\rm d}U(x)\leq \int_{[0,\,t]}(1-G(t-x)){\rm d}U(x)~\sim~ {\tt m}^{-1}\me (\eta\wedge t),\quad t\to\infty,$$ where the asymptotic relation is secured by Theorem 4 in \cite{Sgibnev:1981}. Recall that the function $c_\alpha$ is regularly varying at infinity of index $1/\alpha$. According to Condition {\sc Pert$(\gamma)$} and \eqref{eq:relat},
\begin{equation*}
\lim_{t\to\infty}\frac{\me(\eta\wedge t)}{c_\alpha(t)}=0,
\end{equation*}
which proves \eqref{11}.

Next, we intend to prove \eqref{12}. As has already been mentioned in the proof of Lemma \ref{import}, we can assume that $c_\alpha$ is a nondecreasing function. Integration by parts in \eqref{12} followed by an application of Jensen's inequality yields
$$D(t)=\me \Big|\int_{[0,\,t]}(\nu(t-x)-U(t-x)){\rm d}G(x))\Big|^r \leq \int_{[0,\,t]}\me |\nu(t-x)-U(t-x)|^r {\rm d}G(x).$$ By Theorems 1.1 and 1.4 in \cite{Iksanov+Marynych+Meiners:2016},
\begin{equation}\label{121}
\lim_{t\to\infty}\frac{\me |\nu(t)-{\tt m}^{-1}t|^r}{(c_\alpha(t))^r}=\me |\mathcal{S}_\alpha(1)|^r<\infty.
\end{equation}
Recalling \eqref{eq:star} and \eqref{eq:relat}, we conclude that $$\lim_{t\to\infty}\frac{U(t)-{\tt m}^{-1}t}{c_\alpha(t)}=0.$$ This together with \eqref{121} shows that
$$\lim_{t\to\infty}\frac{\me |\nu(t)-U(t)|^r}{(c_\alpha(t))^r}=\me |\mathcal{S}_\alpha(1)|^r<\infty.$$ Modifying $c_\alpha$ if needed in the right vicinity of $0$ we infer $\me |\nu(t)-{\tt m}^{-1}t|^r\leq A(c_\alpha(t))^r$ for some constant $A>0$ and all $t\geq 0$. With this at hand, $$D(t)\leq \int_{[0,\,t]}\me |\nu(t-x)-U(t-x)|^r {\rm d}G(x)\leq A\int_{[0,\,t]}(c_\alpha(t-x))^r{\rm d}G(x)=O((c_\alpha(t))^r),\quad t\to\infty.$$ We have used monotonicity of $c_\alpha$ for the last equality.
\end{proof}

Lemma \ref{iks2013} is a slight reformulation of Lemma A.5 in \cite{Iksanov:2013}.
\begin{lemma}\label{iks2013}
Let $0\leq a<b<\infty$ and, for each $n\in\mn$, $y_n: [0,\infty)\to [0,\infty)$ be a right-continuous bounded and nondecreasing function. Assume that $\lim_{n\to\infty}x_n=x$ in the $J_1$- or $M_1$-topology on $D$ and that, for each $t\geq 0$, $\lim_{n\to\infty}y_n(t)=y(t)$, where $y:[0,\infty)\to [0,\infty)$ is a bounded continuous function. Then
$$\lim_{n\to\infty}\int_{[a,\,b]}x_n(t){\rm d}y_n(t)=\int_{[a,\,b]}x(t){\rm d}y(t).$$
\end{lemma}

We are ready to prove \eqref{clt21} and \eqref{clt22_sn}.

\subsection{Proof of \eqref{clt21}} This proof proceeds along the lines of the proof of Theorem 3.1 in \cite{Iksanov+Marynych+Samoilenko:2020}.

For $j\in\mn$ and $t\geq 0$, put $D_j(t):={\rm Var}\, N_j(t)$ and $$I_j(t):=\me\bigg(\sum_{r\geq 1}V_{j-1}(t-T_r)\1_{\{T_r\leq t\}}-V_j(t)\bigg)^2$$ with the convention that $V_0(t)=1$ for $t\geq 0$. Our starting point is the recursive formula which is a consequence of \eqref{basic1232}: for $j\geq 2$ and $t\geq 0$,
\begin{eqnarray}\label{aux5}
D_j(t)&=&\me \bigg(\sum_{r\geq 1}\big(N^{(r)}_{j-1}(t-T_r)- V_{j-1}(t-T_r)\big)\1_{\{T_r\leq t\}}\bigg)^2\\&+& \me\bigg(\sum_{r\geq 1}V_{j-1}(t-T_r)
\1_{\{T_r\leq t\}}-V_j(t)\bigg)^2=\int_{[0,\,t]}D_{j-1}(t-y){\rm d}V(y)+I_j(t).\notag
\end{eqnarray}
Starting with $D_1(t)=I_1(t)$ and iterating \eqref{aux5}  we obtain
\begin{equation}\label{recur}
\int_{[0,\,t]}D_{j-1}(t-y){\rm d}V(y)=\sum_{k=1}^{j-1}\int_{[0,\,t]}I_k(t-y){\rm d}V_{j-k}(y),\quad j\geq 2, \quad t\geq 0.
\end{equation}
Our purpose is to show that whenever $j=j(t)\to\infty$ and $j(t)=o(t^{(\gamma-1)/2})$ as $t\to\infty$,
\begin{equation}\label{asymp}
\int_{[0,\,t]}D_{j-1}(t-y){\rm d}V(y)= O\Big(\frac{t^{2j-\gamma}}{(j-2)!(j-1)!{\tt m}^{2j-2}}\Big),\quad t\to\infty.
\end{equation}
We proceed via two steps. First, we show that $I_j$ is upper bounded by a nonnegative and nondecreasing function $h_j$, say, and that the corresponding inequality is valid for all nonnegative arguments. This leads by virtue of \eqref{recur} to a useful inequality for $D_j$ which holds for all nonnegative arguments. Second, we derive an upper bound for both $h_j$ and $D_j$ which is valid for large arguments.

\noindent {\sc Step 1}. Throughout this step it is tacitly assumed that both $j\in\mn$ and $t\geq 0$ are arbitrary.

We start with
\begin{align*}
&\hspace{-1cm}\E \sum_{r\geq 2} \sum_{1\leq i<r}V_{j-1}(t-T_i)\1_{\{T_i\leq t\}}V_{j-1}(t-T_r)\1_{\{T_r\leq t\}}\\
&\leq \E \sum_{i\geq 1}\me\big(V_{j-1}(t-T_i)\1_{\{T_i\leq t\}}\big(V_{j-1}(t-\eta_{i+1}-S_i)\1_{\{\eta_{i+1}\leq t-S_i\}}\\
&+V_{j-1}(t-\eta_{i+2}-\xi_{i+1}-S_i)\1_{\{\eta_{i+2}+\xi_{i+1}\leq t-S_i\}}+\ldots\big)|(\xi_k,\eta_k)_{1\leq i\leq k}\big)\1_{\{S_i\leq t\}}\\
&=\E \sum_{i\geq 1} V_{j-1}(t-T_i)\1_{\{T_i\leq t\}} V_j(t-S_i)\1_{\{S_i\leq t\}}\leq \E \sum_{i\geq 0} V_{j-1}(t-S_i)V_j(t-S_i)\1_{\{S_i\leq t\}}.
\end{align*}
Hence,
\begin{align}
I_j(t)&=\me \sum_{r\geq 1}V^2_{j-1}(t-T_r)\1_{\{T_r\leq t\}}+2\me\sum_{r\geq 2}\sum_{1\leq i<r} V_{j-1}(t-T_i)\1_{\{T_i\leq t\}}V_{j-1}(t-T_r)\1_{\{T_r\leq t\}}-V^2_j(t)\notag\\
& \leq V_{j-1}(t)\me \sum_{r\geq 1}V_{j-1}(t-T_r)\1_{\{T_r\leq t\}}+2\int_{[0,\,t]}V_{j-1}(t-y)V_j(t-y){\rm d}U(y)- V^2_j(t)\notag\\
&= V_{j-1}(t)V_j(t)+2\int_{[0,\,t]}V_{j-1}(t-y)V_j(t-y){\rm d}U(y)- V^2_j(t).\label{259}
\end{align}
Put $\tilde U(t):=\sum_{i\geq 1}\mmp\{S_i\leq t\}$ for $t\geq 0$. Using $\tilde U(t)=U(t)-1$ for $t\geq 0$ and \eqref{eq:star} we conclude that there exists a constant $\tilde c>0$ such that, for all $t\geq 0$,$$|\tilde U(t)-{\tt m}^{-1}t|\leq \tilde c (t+1)^{2-\gamma}.$$ With this at hand integration by parts yields
\begin{multline*}
\int_{[0,\,t]}V_{j-1}(t-y)V_j(t-y){\rm d}U(y)=V_{j-1}(t)V_j(t)+\int_{[0,\,t]}V_{j-1}(t-y)V_j(t-y){\rm d}\tilde U(y)\\
= V_{j-1}(t)V_j(t)+\int_{[0,\,t]}\tilde U(t-y){\rm d}(V_{j-1}(y)V_j(y))\leq (\tilde c (t+1)^{2-\gamma}+1)V_{j-1}(t)V_j(t)+{\tt m}^{-1}\int_0^t V_{j-1}(y)V_j(y){\rm d}y,
\end{multline*}
whence, by \eqref{259},
\begin{multline*}
I_j(t)\leq (2\tilde c (t+1)^{2-\gamma}+3) V_{j-1}(t)V_j(t)+2{\tt m}^{-1}\int_0^t V_{j-1}(y)V_j(y){\rm d}y- V^2_j(t)\\\leq (2\tilde c+3) (t+1)^{2-\gamma} V_{j-1}(t)V_j(t)+2{\tt m}^{-1}\int_0^t V_{j-1}(y)V_j(y){\rm d}y- V^2_j(t) .
\end{multline*}

Invoking \eqref{ineq56789} yields
\begin{align}
&\hspace{-1cm}2{\tt m}^{-1}\int_0^t V_{j-1}(y)V_j(y){\rm d}y\leq 2{\tt m}^{-1}\int_0^t \Big(\frac{y^{j-1}}{(j-1)!{\tt m}^{j-1}}+\sum_{i=0}^{j-2}\binom{j-1}{i}\frac{C^{j-1-i}(y+1)^{(2-\gamma)(j-1-i)+i}}{i!{\tt m}^i}\Big)\\& \times \Big(\frac{y^j}{j!{\tt m}^j}+\sum_{i=0}^{j-1}\binom{j}{i}\frac{C^{j-i}(y+1)^{(2-\gamma)(j-i)+i}}{i!{\tt m}^i}\Big){\rm d}y\notag\\
&\leq\frac{t^{2j}}{(j!)^2{\tt m}^{2j}}+2\frac{(t+1)^{j+1}}{j!{\tt m}^{j+1}}\sum_{i=0}^{j-2}\binom{j-1}{i}\frac{C^{j-1-i}(t+1)^{(2-\gamma)(j-1-i)+i}}{((2-\gamma)(j-1-i)+j+1+i)i!{\tt m}^i}\notag\\&+2 \frac{(t+1)^j}{(j-1)!{\tt m}^j} \sum_{i=0}^{j-1}\binom{j}{i}\frac{C^{j-i}(t+1)^{(2-\gamma)(j-i)+i}}{((2-\gamma)(j-i)+j+i)i!{\tt m}^i}\notag\\
&+ 2\Big( \sum_{i=0}^{j-2}\binom{j-1}{i}\frac{C^{j-1-i}(t+1)^{(2-\gamma)(j-1-i)}+i}{i!{\tt m}^i}\Big)\int_0^t \sum_{i=0}^{j-1}\binom{j}{i}\frac{C^{j-i}(y+1)^{(2-\gamma)(j-i)+i}}{i!{\tt m}^{i+1}}{\rm d}y\notag\\&\leq \frac{t^{2j}}{(j!)^2{\tt m}^{2j}}+ 2 \frac{(t+1)^{j+1}}{(j+1)!{\tt m}^{j+1}}\sum_{i=0}^{j-2}\binom{j-1}{i}\frac{C^{j-1-i}(t+1)^{(2-\gamma)(j-1-i)+i}}{i!{\tt m}^i}\notag\\&+2 \frac{(t+1)^j}{j!{\tt m}^j} \sum_{i=0}^{j-1}\binom{j}{i}\frac{C^{j-i}(t+1)^{(2-\gamma)(j-i)+i}}{i!{\tt m}^i}\notag\\&+2\Big( \sum_{i=0}^{j-2}\binom{j-1}{i}\frac{C^{j-1-i}(t+1)^{(2-\gamma)(j-1-i)}+i}{i!{\tt m}^i}\Big)\Big(\sum_{i=0}^{j-1}\binom{j}{i}\frac{C^{j-i}(t+1)^{(2-\gamma)(j-i)+i+1}}{(i+1)!{\tt m}^{i+1}}\Big)\notag\\&=:\frac{t^{2j}}{(j!)^2{\tt m}^{2j}}+\tilde{f}_j(t).
\label{deffj}
\end{align}
Appealing to \eqref{ineq56789} once again we obtain
\begin{multline*}
V^2_j(t)-\frac{t^{2j}}{(j!)^2{\tt m}^{2j}}=\Big(V_j(t)+\frac{t^j}{j!{\tt m}^j}\Big)\Big(V_j(t)-\frac{t^j}{j!{\tt m}^j}\Big)\geq -\Big(V_j(t)+\frac{t^j}{j!{\tt m}^j}\Big)\sum_{i=0}^{j-1}\binom{j}{i}\frac{C^{j-i}(t+1)^{(2-\gamma)(j-i)+i}}{i!{\tt m}^i}\\
=-\Big(V_j(t)+\frac{t^j}{j!{\tt m}^j}\Big)g_j(t):=-\tilde{g}_j(t).
\end{multline*}
Note that both $\tilde{f}_j$ and $\tilde{g}_j$ are nonnegative nondecreasing functions. Summarizing
\begin{equation}\label{defhj}
I_j(t)\leq (2\tilde c+3) (t+1)^{2-\gamma} V_{j-1}(t)V_j(t)+\tilde{f}_j(t)+\tilde{g}_j(t)=:\tilde{h}_j(t).
\end{equation}
Since $\tilde{h}_j$ is a nondecreasing function, we further infer
$$D_{j-1}(t)=\sum_{k=1}^{j-1}\int_{[0,\,t]} I_k(t-y){\rm d}V_{j-k-1}(y)\leq \tilde{h}_{j-1}(t)
+\sum_{k=1}^{j-2}\tilde{h}_k(t)V_{j-k-1}(t),\quad j\geq 2,\quad t\geq 0.$$

\noindent {\sc Step 2}. Fix now $j\in\mn$ and $s\geq 0$ satisfying $(s+1)^{\gamma-1}\geq 2C{\tt m}j^2$ and let $1\leq k\leq j$. Here, $C$ is the same as in \eqref{ineq56789}. Throughout this step we tacitly assume that all formulae hold true for this range of parameters.

By \eqref{vk}, $$V_{k-1}(s)V_k(s)\leq \frac{4(s+1)^{2k-1}}{(k-1)!k!{\tt m}^{2k-1}}\leq \frac{4(s+1)^{2k-1}}{((k-1)!)^2{\tt m}^{2k-1}}.$$ Next, we show that $$\tilde{f}_k(s)\leq \frac{12C(s+1)^{2k+1-\gamma}}{((k-1)!)^2{\tt m}^{2k-1}}.$$ Indeed, according to \eqref{basic123},
\begin{multline*}
2\frac{(s+1)^{k+1}}{(k+1)!{\tt m}^{k+1}}\sum_{i=0}^{k-2}\binom{k-1}{i}\frac{C^{k-1-i}(s+1)^{(2-\gamma)(k-1-i)+i}}{i!{\tt m}^i}\leq 2 \frac{(s+1)^{k+1}}{(k+1)!{\tt m}^{k+1}} \frac{2C(k-1)(s+1)^{k-\gamma}}{(k-2)!{\tt m}^{k-2}}\\\leq\frac{4C (s+1)^{2k+1-\gamma}}{((k-1)!)^2{\tt m}^{2k-1}}.
\end{multline*}
Analogously, $$2\frac{(s+1)^k}{k!{\tt m}^k} \sum_{i=0}^{k-1}\binom{k}{i}\frac{C^{k-i}(s+1)^{(2-\gamma)(k-i)+i}}{i!{\tt m}^i}\leq \frac{4C(s+1)^{2k+1-\gamma}}{((k-1)!)^2{\tt m}^{2k-1}}.$$
Finally, the third summand in the definition of $\tilde{f}_k$ can be treated as follows:
\begin{multline*}
2\Big(\sum_{i=0}^{k-2}\binom{k-1}{i}\frac{C^{k-1-i}(s+1)^{(2-\gamma)(k-1-i)+i}}{i!{\tt m}^i}\Big)\Big(\sum_{i=0}^{k-1}\binom{k}{i}\frac{C^{k-i}(s+1)^{(2-\gamma)(k-i)+i+1}}{(i+1)!{\tt m}^{i+1}}\Big)\\\leq 2\frac{2C(k-1)(s+1)^{k-\gamma}}{(k-2)!{\tt m}^{k-2}}\frac{2C(s+1)^{k+2-\gamma}}{(k-1)!{\tt m}^k}=\frac{8C^2 (s+1)^{2k+2-2\gamma}}{((k-2)!)^2{\tt m}^{2k-2}}\leq \frac{4C (s+1)^{2k+1-\gamma}}{((k-1)!)^2{\tt m}^{2k-1}}.
\end{multline*}
Here, we have used \eqref{basic123} and \eqref{basic12312} to bound the first and second factor, respectively, and the inequality $(s+1)^{\gamma-1}\geq 2C{\tt m}(k-1)^2$ for the last passage. Finally, $$\tilde{g}_k(s)\leq \frac{6C(s+1)^{2k+1-\gamma}}{((k-1)!)^2 {\tt m}^{2k-1}}$$ by \eqref{vk} and \eqref{basic123}. Summarizing, we have shown that
\begin{equation}\label{ik}
\tilde{h}_k(s)\leq \frac{A(s+1)^{2k+1-\gamma}}{((k-1)!)^2 {\tt m}^{2k-1}},
\end{equation}
where $A:=12+8 \tilde c+18C$.

Further, we obtain, for $s$ satisfying $(s+1)^{\gamma-1}\geq 2\max (c,1){\tt m}j^2=:a_j$, where $c$ is as given in \eqref{lord2},
\begin{align}
D_{j-1}(s)&\leq \tilde{h}_{j-1}(s)+\sum_{k=1}^{j-2}\tilde{h}_k(s)V_{j-k-1}(s)\notag\\
&\leq \frac{A(s+1)^{2j-1-\gamma}}{((j-2)!)^2 {\tt m}^{2j-3}}+2A\sum_{k=1}^{j-2}\frac{(s+1)^{j+k-\gamma}}{(j-k-1)!((k-1)!)^2 {\tt m}^{j+k-2}}\notag\\
&= \frac{A(s+1)^{2j-1-\gamma}}{((j-2)!)^2 {\tt m}^{2j-3}}\Big(1+2 \sum_{k=1}^{j-2}\binom{j-2}{k-1}\frac{(j-2)!}{(k-1)!}\Big(\frac{{\tt m}}{s+1}\Big)^{j-k-1}\Big)\notag\\
&\leq \frac{A(s+1)^{2j-1-\gamma}}{((j-2)!)^2 {\tt m}^{2j-3}}\Big(1+2\frac{{\tt m} j^2}{s+1}\Big(1-\frac{{\tt m} j^2}{s+1}\Big)^{-1}  \Big)\leq \frac{3A(s+1)^{2j-1-\gamma}}{((j-2)!)^2 {\tt m}^{2j-3}}.\label{dj}
\end{align}
Here, the first inequality is just formula \eqref{defhj}, the second inequality is implied by \eqref{vk} and \eqref{ik}, and the third inequality is justified by \eqref{eq:bin}.

Assume now that $j=j(t)\to\infty$ and $j(t)=o(t^{(\gamma-1)/2})$ as $t\to\infty$, so that the inequality $t\geq a_j$ holds true for large enough $t$. We intend to prove \eqref{asymp}.
To this end, we write
\begin{align*}
\int_{[0,\,t]}D_{j-1}(t-y){\rm d}V(y)&=\int_{[0,\,t-a_j]} D_{j-1}(t-y){\rm d}V(y)+\int_{(t-a_j,\,t]} D_{j-1}(t-y){\rm d}V(y)\\
&\leq \frac{3A}{((j-2)!)^2 {\tt m}^{2j-3}}\int_{[0,\,t+1]}(t+1-y)^{2j-1-\gamma}{\rm d}V(y)+ \left(\max_{s\in [0,\,a_j]}D_{j-1}(s)\right)
U(a_j)\\
&\leq \frac{3A(t+1)^{2j-\gamma}}{((j-2)!)^2(2j-\gamma){\tt m}^{2j-2}}+\frac{3Ac(t+1)^{2j+1-2\gamma}}{((j-2)!)^2 {\tt m}^{2j-3}}+\left(\max_{s\in [0,\,a_j]}D_{j-1}(s)\right)U(a_j)
\end{align*}
having utilized \eqref{dj} and $V(x+y)-V(x)\leq U(y)$ for $x,y\in\mr$ (for the proof, see formula (40) in \cite{Bohun etal:2022}) for the first inequality and integration by parts together with \eqref{lord2} for the second. The asymptotic relation $$\frac{3Ac(t+1)^{2j+1-2\gamma}}{((j-2)!)^2 {\tt m}^{2j-3}}=o\Big(\frac{t^{2j-\gamma}}{(j-2)!(j-1)!{\tt m}^{2j-2}}\Big),\quad t\to\infty$$ is a consequence of $j(t)=o(t^{\gamma-1})$ as $t\to\infty$. Using \eqref{vk} and \eqref{dj} for the second inequality below we further obtain
$$
\left(\max_{s\in [0,\,a_j]}D_{j-1}(s)\right)U(a_j)\leq (D_{j-1}(a_j)+V^2_{j-1}(a_j))U(a_j)\leq \Big(\frac{3A(a_j+1)^{2j-1-\gamma}}{((j-2)!)^2 {\tt m}^{2j-3}}+\frac{4(a_j+1)^{2j-2}}{((j-1)!)^2 {\tt m}^{2j-2}}\Big)U(a_j).
$$
By the elementary renewal theorem, with $a=2\max(c,1){\tt m}$,
$$
\frac{(j-2)!(j-1)!{\tt m}^{2j-2}}{t^{2j-\gamma}}\frac{(a_j+1)^{2j-2}}{((j-1)!)^2 {\tt m}^{2j-2}} U(a_j)~\sim~\Big(\frac{aj^2}{t}\Big)^{2j-\gamma}\frac{a j}{{\tt m}}\frac{1}{(aj^2)^{2-\gamma}}~\to~ 0,\quad t\to\infty
$$
because $\lim_{t\to\infty} j^b(aj^2/t)^{2j-2}=0$ for any $b>0$. The last two limit relation hold true whenever $j(t)=o(t^{1/2})$ and particularly under the assumption $j(t)=o(t^{(\gamma-1)/2})$. Analogously,
$$
\frac{(j-2)!(j-1)!{\tt m}^{2j-2}}{t^{2j-\gamma}}\frac{(a_j+1)^{2j-1-\gamma}}{((j-2)!)^2 {\tt m}^{2j-3}} U(a_j)~\sim~\Big(\frac{aj^2}{t}\Big)^{2j-\gamma}j~\to~ 0,\quad t\to\infty.
$$
Thus, $$
\left(\max_{s\in [0,\,a_j]}D_{j-1}(s)\right)U(a_j)=o\Big(\frac{t^{2j-\gamma}}{(j-2)!(j-1)!{\tt m}^{2j-2}}\Big),
$$
and \eqref{asymp} follows.

According to the Cram\'{e}r-Wold device and Markov's inequality, relation \eqref{clt21} follows if we can show that, for any fixed $u>0$,
\begin{equation}\label{eq:aux12}
\lim_{t\to\infty}\frac{((\lfloor j(t)u\rfloor-1)!)^2{\tt m}^{2\lfloor j(t)u\rfloor}}{t^{2\lfloor j(t)u\rfloor-2}c^2_\alpha(t/j(t))}\me
\Big(N_{\lfloor j(t)u\rfloor}(t)-\sum_{r\geq 1}V_{\lfloor j(t)u\rfloor-1}(t-T_r)\1_{\{T_r\leq t\}}\Big)^2=0.
\end{equation}
In view of \eqref{aux5} and \eqref{asymp}, the left-hand side is the big $O$ of
$$\frac{jt^{2-\gamma}}{c^2_\alpha(t/j)}=\frac{j^2}{t^{\gamma-1}}\frac{t/j}{c^2_\alpha(t/j)}.$$ The first factor on the right-hand side is $o(1)$ by assumption. We claim that
\begin{equation}\label{eq:aux10}
\lim_{x\to\infty}x^{-1}c^2_\alpha(x)=\infty,
\end{equation}
so that the second factor on the right-hand side is $o(1)$, too, which proves \eqref{eq:aux12}.

To check \eqref{eq:aux10}, recall that the function $c_\alpha$ is regularly varying at infinity of index $1/\alpha$ which particularly entails
\begin{equation}\label{eq:inter}
\lim_{x\to\infty}c_\alpha(x)=\infty.
\end{equation}
In the case $\alpha\in (1,2)$, the regular variation implies \eqref{eq:aux10}. Assume now that $\alpha=1/2$. Then $c_2$ satisfies $\ell(c_2(x))\sim x^{-1}c^2_2(x)$ as $x\to\infty$, where $\ell$ is a slowly varying diverging to infinity function, see \eqref{eq:slowly varying 2nd moment}. Recalling \eqref{eq:inter} we infer \eqref{eq:aux10} with $\alpha=2$.

\subsection{Proof of \eqref{clt22_sn}} In what follows we write $j$ for $j(t)$. According to the Cram\'{e}r-Wold device, it is enough to show that for any $r\in\mn$, any real $\alpha_1,\ldots, \alpha_r$ and any
$0<u_1<\ldots<u_r<\infty$, as $t\to\infty$,
\begin{equation}\label{fidi}
\sum_{i=1}^r \alpha_i\frac{(\lfloor j u_i\rfloor-1)!{\tt m}^{\lfloor j u_i\rfloor+1/\alpha} Z(j u_i,
t)}{t^{\lfloor ju_i\rfloor-1}c_\alpha(t/j)}~{\overset{{\rm d}}\longrightarrow}~\sum_{i=1}^r \alpha_i u_i\int_0^\infty \mathcal{S}_\alpha(y)e^{-u_i y}{\rm d}y,
\end{equation}
where
$$
Z(ju, t):=\sum_{k\geq 1}V_{\lfloor ju\rfloor-1}(t-T_k)\1_{\{T_k\leq t\}}-V_{\lfloor ju\rfloor}(t),\quad u>0.
$$

For any $u, T>0$ and sufficiently large $t$,
\begin{eqnarray*}
\frac{(\lfloor j u\rfloor-1)!{\tt m}^{\lfloor j u\rfloor+1/\alpha} Z(ju,t)}{t^{\lfloor ju\rfloor -1} c_\alpha(t/j)}&=&\frac{(\lfloor j u\rfloor-1)!{\tt m}^{\lfloor j u\rfloor+1/\alpha}}{t^{\lfloor ju\rfloor -1}c_\alpha(t/j)}
\int_{[0,\,t]}V_{\lfloor ju\rfloor-1}(t-y){\rm d}(N(y)-V(y))\\&=&\frac{(\lfloor j u\rfloor-1)!{\tt m}^{\lfloor j u\rfloor-1}}{t^{\lfloor ju\rfloor-1}} \int_{[0,\,T]}\frac{N(yt/j)-V(yt/j)}
{{\tt m}^{-(\alpha+1)/\alpha}c_\alpha(t/j)}{\rm d}_y(- V_{\lfloor ju\rfloor-1}(t(1-y/j)))\\&+&\frac{(\lfloor j u\rfloor-1)!{\tt m}^{\lfloor j(t)u\rfloor+1/\alpha}}{t^{\lfloor ju\rfloor -1}c_\alpha(t/j)}\int_{(Tt/j,\,t]}(N(y)-V(y)){\rm d}_y(-V_{\lfloor ju\rfloor-1}(t-y)).
\end{eqnarray*}
By Lemma \ref{lem:flt}, $$\Big(\frac{N(ut/j)-V(ut/j)}{{\tt m}^{-(\alpha+1)/\alpha}c_\alpha(t/j)}\Big)_{u\geq 0}~\Longrightarrow~ (\mathcal{S}_\alpha(u))_{u\geq 0}$$
in the $J_1$-topology on $D$ if $\alpha=2$ and in the $M_1$-topology on $D$ if $\alpha\in (1,2)$. Here, we have used the assumption $t/j(t)\to\infty$. By Skorokhod's
representation theorem there exist versions $\widehat N_t$ and $\widehat{\mathcal{S}}_\alpha$ of $((N(ut/j)-V(ut/j))/({\tt m}^{-(\alpha+1)/\alpha}c_\alpha(t/j))_{u\geq 0}$ and $\mathcal{S}_\alpha$, respectively such that
\begin{equation}\label{cs}
\lim_{t\to\infty}\widehat N_t(y)=\widehat{\mathcal{S}}_\alpha(y)\quad\text{a.s.}
\end{equation}
in the $J_1$-topology on $D$ if $\alpha=2$ and in the $M_1$-topology on $D$ if $\alpha\in (1,2)$. In view of \eqref{impo121}, $$\lim_{t\to\infty} \frac{(\lfloor j u\rfloor-1)!{\tt m}^{\lfloor j u\rfloor-1}}{t^{\lfloor ju\rfloor-1}}V_{\lfloor ju\rfloor-1}(t(1-y/j))= e^{-uy},\quad t\to\infty$$ for each fixed $y\geq 0$. By Lemma \ref{iks2013}, this in combination with \eqref{cs}
yields
$$\lim_{t\to\infty}\sum_{i=1}^r \alpha_i\frac{(\lfloor j u\rfloor-1)!{\tt m}^{\lfloor j u\rfloor-1}}{t^{\lfloor ju\rfloor-1}}  \int_0^T \widehat N_t(y){\rm d}_y(-V_{\lfloor ju\rfloor-1}(t(1-y/j)))=\sum_{i=1}^r \alpha_i
u_i\int_0^T \widehat{\mathcal{S}}_\alpha(y)e^{-u_iy}{\rm d}y\quad \text{a.s.}$$ and
thereupon
$$
\sum_{i=1}^r \alpha_i \frac{(\lfloor j u\rfloor-1)!{\tt m}^{\lfloor j u\rfloor-1}}{t^{\lfloor ju\rfloor-1}}\int_{[0,\,T]}\frac{N(yt/j)-V(yt/j)}{{\tt m}^{-(\alpha+1)/\alpha}c_\alpha(t/j)}{\rm d}_y(-V_{\lfloor ju\rfloor-1}(t(1-y/j)))~{\overset{{\rm
d}}\longrightarrow} ~\sum_{i=1}^r \alpha_i u_i\int_0^T \mathcal{S}_\alpha(y)e^{-u_iy}{\rm
d}y,
$$
as $t\to\infty$. Since $\lim_{T\to\infty}\sum_{i=1}^r \alpha_i u_i \int_0^T
\mathcal{S}_\alpha(y)e^{-u_iy}{\rm d}y=\sum_{i=1}^r \alpha_i u_i \int_0^\infty
\mathcal{S}_\alpha (y)e^{-u_iy}{\rm d}y$ a.s.\ we are left with proving that
$$\lim_{T\to\infty}{\lim\sup}_{t\to\infty}\,\mmp\bigg\{\bigg|\sum_{i=1}^r\alpha_i \frac{(\lfloor j u_i\rfloor-1)!{\tt m}^{\lfloor j u\rfloor+1/\alpha}}{t^{\lfloor ju_i\rfloor-1}c_\alpha(t/j)}\int_{(Tt/j,\,t]}(N(y)-V(y)){\rm d}(-V_{\lfloor ju\rfloor-1}(t-y))\bigg|>\varepsilon\bigg\}=0$$ for all $\varepsilon>0$. By Lemma \ref{aux123}, $\me |N(y)-V(y)| \sim {\tt m}^{-(\alpha+1)/\alpha} \me |\mathcal{S}_\alpha(1)|c_\alpha(y)$ as $y\to\infty$. With this at hand, the last limit relation follows from Markov's inequality and \eqref{aux0} The proof of \eqref{clt22_sn} is complete.

\vspace{5mm}

\noindent {\bf Acknowledgement}. The present work was supported by the National Research Foundation of Ukraine (project 2020.02/0014 'Asymptotic regimes of perturbed random walks: on the edge of modern and classical probability').

\end{document}